\documentclass[letterpaper,conference]{ieeeconf}

\IEEEoverridecommandlockouts

\usepackage{graphicx}
\usepackage{amsmath}
\usepackage{amssymb}
\usepackage{url}
\usepackage[compress]{cite}

\renewcommand{\theenumi}{\textit{\roman{enumi}}}

\newtheorem{lemma}{\sc Lemma}
\newtheorem{theorem}[lemma]{\sc Theorem}
\newtheorem{corollary}[lemma]{\sc Corollary}
\newtheorem{remark}{\sc Remark}
\newtheorem{assumption}{\sc Assumption}
\newtheorem{definition}{\sc Definition}

\renewcommand{\matrix}[2]{\left[\begin{array}{#1} #2 \end{array}\right]}

\DeclareMathOperator*{\argmin}{arg\,min}
\DeclareMathOperator*{\diag}{diag}
\DeclareMathOperator*{\trace}{Tr}

\renewcommand{\footnoterule}{%
  \kern -7pt
  \hrule width 0.3\textwidth height .5pt
  \kern 2pt
}

\title{Quadratic Gaussian Privacy Games}
\author{Farhad Farokhi, Henrik Sandberg, Iman Shames, and Michael Cantoni\thanks{F.~Farokhi, I.~Shames, and M.~Cantoni are with the Department of Electrical and Electronic Engineering, The University of Melbourne, Parkville, Victoria 3010, Australia. Their work was supported by a McKenzie Fellowship, the Australian Research Council (LP130100605), and Rubicon Water~Pty~Ltd. Emails: \{ffarokhi,ishames,cantoni\}@unimelb.edu.au}\thanks{H. Sandberg is with the ACCESS Linnaeus Center, School of Electrical Engineering, KTH Royal Institute of Technology, SE-100 44 Stockholm, Sweden. Email: hsan@kth.se}}

\begin{document}

\maketitle

\begin{abstract} A game-theoretic model for analysing the effects of privacy on strategic communication between agents is devised. In the model, a sender wishes to provide an accurate measurement of the state to a receiver while also protecting its private information (which is correlated with the state) private from a malicious agent that may eavesdrop on its communications with the receiver. A family of nontrivial equilibria, in which the communicated messages carry information, is constructed and its properties are studied. 
\end{abstract}

\section{Introduction}
Modern infrastructures, such as smart grids and intelligent transportation systems, are typically composed of smaller entities that cooperate in providing services whilst competing to generate revenue. These entities wish to retain a certain level of autonomy and privacy, which renders the task of guaranteeing a reasonable level of performance arduous. A new framework is required for modelling, understanding, and mitigating the challenges arising from the need for privacy and autonomy. To this aim, in this paper, we develop a game-theoretic setup for modelling and analysing the effects of privacy on strategic\footnote{The term ``strategic'' in ``strategic communication'' refers to that the participants (the sender, the receiver, the eavesdropper, etc) can identify their objectives and find the best policy for achieving these objectives in response to the actions of the other parties. } communication.

Specifically, we consider strategic communication between a sender and a receiver when a malicious agent can eavesdrop on their transmitted messages. The sender has the intention to provide accurate measurements of the state to the receiver, however, at the same time, it would like to hide its private information (which is correlated to the state-to-be-estimated) from the malicious agent\footnote{Note that the receiver and the malicious agent need not be separate entities. For instance, in many cases, we would like to provide some information to for-profit organizations in return for services, however, we wish to retain the rest of our personal data private to avoid being scammed through, among various means, targeted advertisement.}. We present a game-theoretic framework for modelling the interaction between these agents. Subsequently, we define an equilibrium as a tuple of policies from which no one has any incentive to unilaterally deviate. We study the properties of the equilibrium as function of the privacy ratio, a parameter that reflects the balance
 between the sender's intention for providing an accurate measurement of the state to the receiver and its desire for privacy against the malicious agent. 

The problem that is considered in this paper is partly related to studies in differential privacy and its application in  control and estimation~\cite{Dwork2006,Dwork2008,6606817,Huang2014Differential}.  In those studies, the communicated messages are typically perturbed by noise to ensure differential privacy. In this paper, we take a completely different approach to model and to understand the balance between privacy and the underlying objective of communication, that is, providing useful information to a less-informed third party. This framework results in different insights. For instance, in turns out that, in some cases, we do not need to add noise to the communicated signals at the equilibrium of the introduced privacy game. Furthermore, this paper provides an average-type privacy whereas the approach in the differential privacy literature gives a worst-case analysis (using information-theoretic tools one can bound the amount of the total leaked information).

We should note that the model and the results in this paper are, particularly, related to the cheap-talk literature~\cite{crawford1982strategic,Sobel2009signaling,farrell1996cheap}. A cheap-talk game is a strategic communication game in which well-informed players transmits messages to a receiver that makes a final decision regarding social welfare of all the players. Recently, these results were utilized in~\cite{6859123} to develop a framework for strategic communication between self-interested sensors and an estimator. In this paper, we utilize those ideas to develop a specific family of cheap-talk games in which the conflict-of-interest between sensors and the receiver is caused by the desire for privacy. Subsequently, this family of games is analysed using the tools developed in~\cite{6859123}.

The rest of the paper is organized as follows. In Section~\ref{sec:problem}, we introduce the privacy game and define the equilibrium for the game. We determine a family of equilibria and study its properties in Section~\ref{sec:results}. Finally, we present a numerical example in Section~\ref{sec:example} and conclude the paper in Section~\ref{sec:conclusions}.

\begin{figure}[t]
\centering
\includegraphics[width=0.8\columnwidth]{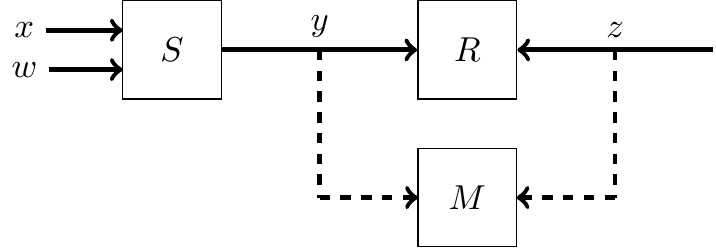}
\caption{\label{figure:communication_structure} The communication structure between the strategic sensor $S$, other sources of information $z$, the malicious agent $M$, and the receiver $R$.}
\end{figure}


\section{Problem Formulation} \label{sec:problem}
We consider the communication structure (see Fig.~\ref{figure:communication_structure}) in which a sender $S$ transmits a message $y\in\mathbb{R}^{n_y}$ over a real channel to a receiver $R$. The sender has the best intention to provide an accurate measurement of a variable $x\in\mathbb{R}^{n_x}$ to the receiver. However, it also wishes to keep its private information $w\in\mathbb{R}^{n_w}$ hidden from any malicious agent who can tap the channel and eavesdrop on its transmitted message. This is certainly not a trivial task if $x$ and $w$ are correlated (as, in such cases, measurements of $x$ carry ``some information'' about the realization of $w$). We also assume that the receiver and the malicious agent can, in addition to the message provided by the sender,  access other sources of information (e.g., other senders, physical sensors, etc) to receive the measurement $z\in\mathbb{R}^{n_z}$. Throughout the paper, we use the notation $V_{xy}$ to denote $\mathbb{E}\{xy^\top\}$ for any two random variables $x$ and $y$.  We make the following assumption.

\begin{assumption} \label{assum:Gaussian} $x,w,z$ are jointly distributed Gaussian random variables with zero mean and a positive-definite joint covariance matrix.
\end{assumption}

The timing of the game is as follows. First, the random variables $x,w,z$ are realized. Then, the sender determines and transmits its message $y$. Note that, at this point, we assume that the sender has access to the realization of $z$, however, as will see, at the equilibrium, the sender does not utilize the realization (but uses the probability distribution). Subsequently, the receiver constructs its best estimate of the state $x$. At the same time, a malicious agent constructs the best estimate of the private information $w$ of the sender from the available information $y$ and $z$. Finally, all the cost functions are realized. Note that all the distributions are assumed to be common knowledge in the game (i.e., each player knows them, each player knows that the other players know them, and so on). Let us introduce each player in the game properly.

\subsection{Receiver}
The receiver constructs its best estimate of the state, denoted by $\hat{x}\in\mathbb{R}^{n_x}$ according to the conditional probability density distribution $\psi(\cdot|y,z)$. Therefore, for all Lebesgue-measurable sets $\mathcal{X}\subseteq \mathbb{R}^{n_x}$, we have
\begin{align*}
\mathbb{P}\{\hat{x}\in\mathcal{X}|y,z\}=\int_{\xi\in\mathcal{X}} \psi(\xi|y,z) \mathrm{d}\xi.
\end{align*}
Let $\Psi$ denote the set of all such conditional probability density distributions (i.e., the set of all non-negative Lebesgue-measurable functions with integral over $\mathbb{R}^{n_x}$ equal to one). The receiver is interested in having a small estimation error for the state $\mathbb{E}\{\|x-\hat{x}\|_2^2\}$. 


\subsection{Malicious Agent}
 The malicious agent constructs its best estimate $\hat{w}$ according to the conditional probability density distribution $\phi(\cdot|y,z)$. Therefore, similarly, for all Lebesgue-measurable sets $\mathcal{W}\subseteq \mathbb{R}^{n_w}$, we get
\begin{align*}
\mathbb{P}\{\hat{w}\in\mathcal{W}|y,z\}=\int_{\xi\in\mathcal{W}} \phi(\xi|y,z) \mathrm{d}\xi.
\end{align*}
The set of all such conditional distributions is denoted by $\Phi$. What distinguishes the sets $\Psi$ and $\Phi$ is the dimensions of the vectors $\hat{x}$ and $\hat{w}$.  The cost function of the malicious agent is $\mathbb{E}\{ \|w-\hat{w}\|_2^2\}$. 
 

\subsection{Sender}
The sender employs a conditional probability density distribution $\gamma(\cdot|x,w)$ for all $x\in\mathbb{R}^{n_x}$ and $w\in\mathbb{R}^{n_w}$, which gives that 
\begin{align*}
\mathbb{P}\{y\in\mathcal{Y}|x,w\}=\int_{\xi\in\mathcal{Y}} \gamma(\xi|x,w) \mathrm{d}\xi,
\end{align*}
where $\mathcal{Y}\subseteq \mathbb{R}^{n_y}$ is any Lebesgue-measurable set. The set of all such conditional probability density distributions is denoted by $\Gamma$. The cost of the sender is $\mathbb{E}\{\|x-\hat{x}\|_2^2\}-\delta \mathbb{E}\{ \|w-\hat{w}\|_2^2\}$, where $\delta\in\mathbb{R}_{\geq 0}$ is the privacy ratio. As $\delta$ increases, the sender puts more emphasis on keeping its privacy instead of helping the receiver by proving a useful message for the state estimation to the receiver.

\subsection{Equilibrium}
With these definitions at hand, we are now ready to define the equilibrium of the game. However, first, we need to define some useful notations. In that what follows, we use $\Psi^\Gamma$ and $\Phi^\Gamma$ to denote the sets of all mappings from $\Gamma$ to, respectively, $\Psi$ and $\Phi$. Hence, for any $\psi\in\Psi^\Gamma$ and any given $\gamma$, $\psi(\gamma)$ is a conditional probability density distribution in $\Psi$. We use the notation $[\psi(\gamma)](\hat{x}|y,z)$ to distinguish between the arguments of $\psi$ and $\psi(\gamma)$. The same can be said about the entries of $\Phi^\Gamma$.


\begin{definition} \label{def:eq} A tuple $(\psi^*,\phi^*,\gamma^*)\in\Psi^\Gamma\times \Phi^\Gamma\times \Gamma$ is an equilibrium of the privacy game with privacy ratio $\delta$ if 
\begin{subequations}
\begin{align}
\psi^*&\in\argmin_{\psi\in\Psi^\Gamma} \mathbb{E}\bigg\{\int_{\hat{x}\in\mathbb{R}^{n_x}}\int_{y\in\mathbb{R}^{n_y}}  \|x-\hat{x}\|_2^2[\psi(\gamma^*)](\hat{x}|y,z)\nonumber\\
&\hspace{1.5in}\times\gamma^*(y|x,w)\mathrm{d}y\mathrm{d}\hat{x}\bigg\},\label{eqn:psi*} \\
\phi^*&\in\argmin_{\phi\in\Phi^\Gamma} \mathbb{E}\bigg\{\int_{\hat{w}\in\mathbb{R}^{n_w}}\int_{y\in\mathbb{R}^{n_y}} \|w-\hat{w}\|_2^2[\phi(\gamma^*)](\hat{w}|y,z)\nonumber\\
&\hspace{1.5in}\times\gamma^*(y|x,w)\mathrm{d}y\mathrm{d}\hat{w}\bigg\},\label{eqn:phi*} \\
\gamma^*&\in\argmin_{\gamma\in\Gamma} \mathbb{E}\bigg\{ \int_{\hat{x}\in\mathbb{R}^{n_x}}\int_{y\in\mathbb{R}^{n_y}}  \|x-\hat{x}\|_2^2[\psi^*(\gamma)](\hat{x}|y,z)\nonumber\\
&\hspace{1.5in}\times\gamma(y|x,w)\mathrm{d}y\mathrm{d}\hat{x}\nonumber\\
&\hspace{.48in}-\delta \hspace{-.05in}\int_{\hat{w}\in\mathbb{R}^{n_w}}\int_{y\in\mathbb{R}^{n_y}} \|w-\hat{w}\|_2^2[\phi^*(\gamma)](\hat{w}|y,z)\nonumber\\
&\hspace{1.5in}\times\gamma(y|x,w)\mathrm{d}y\mathrm{d}\hat{w} \bigg\}. \label{eqn:gamma*}
\end{align}
\end{subequations}
\end{definition} 


In some cases, the policies of the malicious agent and the receiver are set by the same entity. For instance, we may want to provide some information to a for-profit organization in return for services, however, we wish to retain the rest of our personal data private to avoid being scammed through targeted advertisement from the same organization. The following definition capture such scenarios. 

\begin{definition} \label{def:eq:coalition} A tuple $(\psi^*,\phi^*,\gamma^*)\in\Psi^\Gamma\times \Phi^\Gamma\times \Gamma$ is an equilibrium of the privacy game with privacy ratio $\delta$ if $\gamma^*$ satisfies~\eqref{eqn:gamma*} and
\begin{align}
(\psi^*&,\phi^*)\in\hspace{-.15in}\argmin_{(\psi,\phi)\in\Psi^\Gamma\times \Phi^\Gamma}\hspace{-.12in} \mathbb{E}\bigg\{\int_{\hat{x}\in\mathbb{R}^{n_x}}\int_{y\in\mathbb{R}^{n_y}}\hspace{-.2in}  \|x-\hat{x}\|_2^2[\psi(\gamma^*)](\hat{x}|y,z)\nonumber\\
&\hspace{1.5in}\times\gamma^*(y|x,w)\mathrm{d}y\mathrm{d}\hat{x}\nonumber\\
&\hspace{0.7in}+\vartheta\int_{\hat{w}\in\mathbb{R}^{n_w}}\int_{y\in\mathbb{R}^{n_y}}\hspace{-.2in} \|w-\hat{w}\|_2^2[\phi(\gamma^*)](\hat{w}|y,z)\nonumber\\
&\hspace{1.5in}\times\gamma^*(y|x,w)\mathrm{d}y\mathrm{d}\hat{w}\bigg\} \label{eqn:psi*phi*}
\end{align}
for a given constant $\vartheta>0$.
\end{definition}

\begin{theorem} \label{tho:equivalence} A tuple $(\psi^*,\phi^*,\gamma^*)\in\Psi^\Gamma\times \Phi^\Gamma\times \Gamma$ is an equilibrium of the privacy game with privacy ratio $\delta$ in the sense of Definition~\ref{def:eq:coalition} if and only if it is an equilibrium in the sense of Definition~\ref{def:eq}. 
\end{theorem} 

\begin{proof} The first term under expectation in~\eqref{eqn:psi*phi*} is only a function of $\psi$ and the second term is only a function of $\phi$. Therefore, we can separate this optimization problem into two parts as in~\eqref{eqn:psi*} and~\eqref{eqn:phi*}.
\end{proof}

Theorem~\ref{tho:equivalence} shows that there is no loss of generality in assuming that the receiver and the malicious agent are two separate entities as in Definition~\ref{def:eq}. 

\begin{remark}
The equilibrium in Definition~\ref{def:eq:coalition} (and equivalently that of Definition~\ref{def:eq}) is a Stackelberg equilibrium rather than a Nash equilibrium. This is because, in the definition of a Nash equilibrium, the players are not allowed to change their policy (they select one and fix it in response to all that the others can do). However, in this definition, clearly, both the receiver and the malicious agent are able to set their policies (i.e. estimate generating conditional densities $\psi(\gamma)$ and $\phi(\gamma)$) in response to the selected policy of the sender $\gamma$. 
\end{remark}

In the next section, we construct a family of privacy-game equilibria.

\begin{figure*}[!t]
\begin{align} 
\mathbb{E}\{\|x-\hat{x}\|_2\}-\delta\mathbb{E}&\{\|w-\hat{w}\|_2\}=
\trace(V_{xx}-V_{xz}V_{zz}^{-1}V_{zx})-\delta \trace(V_{ww}-V_{wz}V_{zz}^{-1}V_{zw}) \nonumber \\
&+\trace\left(
\matrix{c}{V_{xy} \\ V_{wy} \\ V_{zy} }^\top  
\matrix{ccc}{ 
-I & 0 & V_{xz}V_{zz}^{-1} \\
0 & \delta I &  -\delta V_{wz}V_{zz}^{-1} \\
V_{zz}^{-1} V_{zx}& -\delta V_{zz}^{-1}V_{zw} & -V_{zz}^{-1}(V_{zx}V_{xz}-\delta V_{zw}V_{wz} )V_{zz}^{-1}
}
\matrix{c}{V_{xy} \\ V_{wy} \\ V_{zy} } \right).
\label{eqn:longeqn:proof:1}
\end{align}
\hrule
\end{figure*}

\section{Main Results} \label{sec:results}
As any other cheap-talk game, this game posses an array of non-informative trivial equilibria, which are known as \textit{babbling equilibria} in the economics literature~\cite{Sobel2009signaling}. The next theorem constructs these babbling equilibria in the privacy game in line with Definition~\ref{def:eq}.

\begin{theorem} \label{tho:babbling} Let $\gamma^*(y|x,w)=\rho_y(y)$, where $\rho_y(\cdot)$ is an arbitrary probability density distribution. Moreover, let $\psi^*$ and $\phi^*$ be, respectively, selected such that 
$\hat{x}=\mathbb{E}\{x|z\}$ and $\hat{w}=\mathbb{E}\{w|z\}$ almost surely. Then, $(\psi^*,\phi^*,\gamma^*)$ is an equilibrium of the privacy game with any privacy ratio.
\end{theorem}

\begin{proof} Since the sender does not provide any viable information about the state-to-be-estimated and so $y$ and $x$ are statistically independent, the best policy of the receiver and the malicious agent is to ignore the message $y$. Similarly, if the receiver and the malicious agent are ignoring the sender's message, its cost function does not change by the choice of its policy (i.e., the conditional density function for selecting its message) and so it might as well not provide any viable information.
\end{proof}


At these equilibria, all the players purely randomize and disregard all the available information to them. We are not interested in any of these equilibria. In the next theorem, we construct a nontrivial equilibrium for the privacy game.

\begin{theorem} \label{tho:1} For any $\delta\in\mathbb{R}_{\geq 0}$, there exists an equilibrium $(\psi^*,\phi^*,\gamma^*)\in\Psi^\Gamma\times \Phi^\Gamma\times \Gamma$ in which
\begin{enumerate}
\item The receiver employs the policy $\psi^*$ such that
\begin{align*}
\mathbb{P}\left\{\hat{x}=\matrix{cc}{V_{xy} & V_{xz}}\matrix{cc}{V_{yy} & V_{yz} \\ V_{zy} & V_{zz}}^{-1}\matrix{c}{y \\ z}\right\}=1;
\end{align*}
\item The malicious agent employs the policy $\phi^*$ such that
\begin{align*}
\mathbb{P}\left\{\hat{w}=\matrix{cc}{V_{wy} & V_{wz}}\matrix{cc}{V_{yy} & V_{yz} \\ V_{zy} & V_{zz}}^{-1}\matrix{c}{y \\ z}\right\}=1;
\end{align*}
\item The sender employs the policy $\gamma^*$ such that
\begin{align*}
y=K_x^* x+K_w^* w+v,
\end{align*}
where
\begin{align*}
\matrix{c}{K_x^{*\top} \\ K_w^{*\top} }\in\hspace{-.11in}\argmin_{K^\top\hspace{-.03in}\in\mathbb{R}^{(n_x+n_y)\times n_y}}\hspace{-.2in} & \,\, \trace\left( K \Xi\matrix{cc}{-I & 0 \\ 0 & \delta I}\Xi K^\top\right), \\
\mathrm{s.t.} \hspace{.1in} & \,\, K \Xi K^\top\leq I,
\end{align*}
with
\begin{align*}
\Xi&=\matrix{cc}{V_{xx} & V_{xw} \\ V_{wx} & V_{ww} }-
\matrix{c}{V_{xz} \\ V_{wz}}^\top V_{zz}^{-1}\matrix{c}{V_{xz} \\ V_{wz}},
\end{align*}
and $v\in\mathbb{R}^n$ is a Gaussian random variable independent of $x,w,z$ with zero mean and variance
$$
V_{vv}=I-\matrix{cc}{K_x^{*} & K_w^{*} }\Xi\matrix{c}{K_x^{*\top} \\ K_w^{*\top} }.
$$
\end{enumerate}
Furthermore, for all $\kappa\in\mathbb{R}^{n_y\times n_y}$ such that $\det(\kappa)\neq 0$, $(\psi^*,\phi^*,\kappa\gamma^*)$ constitutes an equilibrium.
\end{theorem}

\begin{proof} First, we fix $\psi^*$ and $\phi^*$ as in the statement of the theorem and show that the presented $\gamma^*$ is the best response of the sender. Let us assume that $V_{yy}-V_{yz}V_{zz}^{-1}V_{zy}=I$. This is without loss of generality as the receiver can always scale the communicated message to ensure this is satisfied. 
 To show this, assume that $V_{yy}-V_{yz}V_{zz}^{-1}V_{zy}=D\in\mathcal{S}_{++}^{n_y}$, where $\mathcal{S}_{++}^{n_y}$ denotes the set of positive-definite matrix in $\mathbb{R}^{n_y\times n_y}$. Note that $V_{yy}-V_{yz}V_{zz}^{-1}V_{zy}\in\mathcal{S}_{++}^{n_y}$ follows from Schur complement and Assumption~\ref{assum:Gaussian}. Now, the receiver can simply use $\tilde{y}=D^{-1/2}y$, which gives $V_{\tilde{y}\tilde{y}}-V_{\tilde{y}z}V_{zz}^{-1}V_{z\tilde{y}}=D^{-1/2}DD^{-1/2}=I$. (Since the transformation between $y$ and $\tilde{y}$ is one-to-one, it does not create any information loss.) Under this assumption, using the matrix inversion lemma, we get
\begin{align*}
\matrix{cc}{V_{yy} & V_{yz} \\ V_{zy} & V_{zz}}^{-1}\hspace{-.16in}
&=\hspace{-.04in}\matrix{cc}{I & -V_{yz}V_{zz}^{-1} \\ -V_{zz}^{-1}V_{zy} & V_{zz}^{-1}V_{zy}V_{yz}V_{zz}^{-1}+V_{zz}^{-1} }\hspace{-.03in}.
\end{align*}
Note that
\begin{align*}
&\mathbb{E}\{\|x-\hat{x}\|_2\}\\
&=\mathbb{E}\bigg\{\bigg\|x-\matrix{cc}{V_{xy} & V_{xz}}\matrix{cc}{V_{yy} & V_{yz} \\ V_{zy} & V_{zz}}^{-1}\matrix{c}{y \\ z}\bigg\|_2\bigg\}\\
&=\trace\bigg(V_{xx}\hspace{-.03in}-\hspace{-.03in}\matrix{cc}{V_{xy} & V_{xz}}\hspace{-.04in}\matrix{cc}{V_{yy} & V_{yz} \\ V_{zy} & V_{zz}}^{-1}\hspace{-.04in}\matrix{c}{V_{yx} \\ V_{zx}}\bigg)\\
&=\trace(V_{xx}-V_{xy}V_{yx}+V_{xy}V_{yz}V_{zz}^{-1}V_{zx}+V_{xz}V_{zz}^{-1}V_{zy}V_{yx}
\\&\hspace{.1in}-V_{xz} (V_{zz}^{-1}V_{zy}V_{yz}V_{zz}^{-1}+V_{zz}^{-1})V_{zx})\\
&=\trace(V_{xx}-V_{xz}V_{zz}^{-1}V_{zx})\hspace{-.03in}+\hspace{-.03in}\trace(-V_{yx}V_{xy}+V_{yz}V_{zz}^{-1}V_{zx}V_{xy}
\\&\hspace{.1in}+V_{yx}V_{xz}V_{zz}^{-1}V_{zy}-V_{yz}V_{zz}^{-1}V_{zx}V_{xz} V_{zz}^{-1}V_{zy}).
\end{align*}
Similarly, we have
\begin{align*}
&\mathbb{E}\{\|w-\hat{w}\|_2\}\\
&=\mathbb{E}\bigg\{\bigg\|w-\matrix{cc}{V_{wy} & V_{wz}}\matrix{cc}{V_{yy} & V_{yz} \\ V_{zy} & V_{zz}}^{-1}\matrix{c}{y \\ z}\bigg\|_2\bigg\}\\
&=\trace(V_{ww}-V_{wz}V_{zz}^{-1}V_{zw})+\trace(-V_{yw}V_{wy}\\
&\hspace{.1in}+V_{yz}V_{zz}^{-1}V_{zw}V_{wy}+V_{yw}V_{wz}V_{zz}^{-1}V_{zy}\\
&\hspace{.1in}-V_{yz}V_{zz}^{-1}V_{zw}V_{wz} V_{zz}^{-1}V_{zy}).
\end{align*}
Therefore, we get the equality in~\eqref{eqn:longeqn:proof:1} on top of the page.
Notice that, in this cost function, we have
\begin{align*}
&\matrix{ccc}{ 
-I & 0 & V_{xz}V_{zz}^{-1} \\
0 & \delta I &  \delta V_{wz}V_{zz}^{-1} \\
V_{zz}^{-1} V_{zx}& \delta V_{zz}^{-1}V_{zw} & -V_{zz}^{-1}(V_{zx}V_{xz}-\delta V_{zw}V_{wz} )V_{zz}^{-1}
}\\
&=\hspace{-.04in}
\matrix{ccc}{I & 0 & -V_{xz}V_{zz}^{-1} \\ 0 & I & -V_{wz}V_{zz}^{-1}}^{\hspace{-.03in}\top}
\hspace{-.06in}\matrix{cc}{-I & 0 \\ 0 & \delta I}\hspace{-.06in}
\matrix{ccc}{I & 0 & -V_{xz}V_{zz}^{-1} \\ 0 & I & -V_{wz}V_{zz}^{-1}}\hspace{-.03in}.
\end{align*}
Let us call this matrix $Z$. Moreover, we have
\begin{align} \label{eqn:inequality:1}
\mathbb{E}\left\{ \matrix{c}{y \\ x \\ w \\ z} \matrix{c}{y \\ x \\ w \\ z}^\top\right\}
&\hspace{-.03in}=\hspace{-.03in}\matrix{cccc}{
V_{yy} & V_{yx} & V_{yw} &  V_{yz} \\
V_{xy} & V_{xx} & V_{xw} &  V_{xz} \\
V_{wy} & V_{wx} & V_{ww} &  V_{wz} \\
V_{zy} & V_{zx} & V_{zw} &  V_{zz}
 } \hspace{-.03in}\geq\hspace{-.03in} 0.
\end{align}
Using Schur complement, the inequality in~\eqref{eqn:inequality:1} can be rewritten as
\begin{align} \label{eqn:inequality:2}
V_{yy}-
\matrix{c}{V_{xy} \\ V_{wy} \\ V_{zy}}^\top
\matrix{cccc}{
V_{xx} & V_{xw} &  V_{xz} \\
V_{wx} & V_{ww} &  V_{wz} \\
V_{zx} & V_{zw} &  V_{zz} }^{-1}
\matrix{c}{V_{xy} \\ V_{wy} \\ V_{zy}}
\geq 0.
\end{align}
Recalling that $V_{yy}-V_{yz}V_{zz}^{-1}V_{zy}=I$, the inequality in~\eqref{eqn:inequality:2} can be posed as
\begin{align} \label{eqn:inequality:3}
I-&
\matrix{c}{V_{xy} \\ V_{wy} \\ V_{zy}}^\top
\left(
\matrix{cccc}{
0 & 0 & 0\\
0 & 0 & 0 \\
0 & 0 & V_{zz}^{-1} }\right.\nonumber\\&+\left.
\matrix{cccc}{
V_{xx} & V_{xw} &  V_{xz} \\
V_{wx} & V_{ww} &  V_{wz} \\
V_{zx} & V_{zw} &  V_{zz} }^{-1}
\right)
\matrix{c}{V_{xy} \\ V_{wy} \\ V_{zy}}
\geq 0.
\end{align}
Following the same argument as in the proof of Theorem~2.2 in~\cite{FarokhiTAC2014Submitted}, we can show that the inequality in~\eqref{eqn:inequality:3} holds if and only if
\begin{align*}
\matrix{c}{V_{xy} \\ V_{wy} \\ V_{zy}}^\top
Q
\matrix{c}{V_{xy} \\ V_{wy} \\ V_{zy}}
\leq I,
\end{align*}
where
\begin{align*}
Q=&
\matrix{ccc}{
I & 0 & -V_{xz}V_{zz}^{-1} \\
0 & I & -V_{wz}V_{zz}^{-1}}^\top
\\&\times
\left(
\matrix{cc}{V_{xx} & V_{xw} \\ V_{wx} & V_{ww} }-
\matrix{c}{V_{xz} \\ V_{wz}}V_{zz}^{-1}\matrix{c}{V_{xz} \\ V_{wz}}
\right)^{-1}
\\&\times
\matrix{ccc}{
I & 0 & -V_{xz}V_{zz}^{-1} \\
0 & I & -V_{wz}V_{zz}^{-1}}.
\end{align*}
Thus, the best response of the sender is given by solving the optimization problem
\begin{subequations} \label{eqn:optim:1}
\begin{align}
\min_{V_{xy},V_{wy},V_{zy}} & \trace\left( \matrix{c}{V_{xy} \\ V_{wy} \\ V_{zy}}^\top Z \matrix{c}{V_{xy} \\ V_{wy} \\ V_{zy}}\right) \\
\mathrm{s.t.} \hspace{.2in} & \matrix{c}{V_{xy} \\ V_{wy} \\ V_{zy}}^\top Q \matrix{c}{V_{xy} \\ V_{wy} \\ V_{zy}}\leq I.
\end{align}
\end{subequations}
It remains to show that an affine policy can be optimal in this sense. Let
$$
y=K_x x+K_w w+K_z z+v,
$$
where $v$ is a zero-mean Gaussian random variable, which is independent of $x,w,z$. Simple calculations show that
\begin{align*}
\matrix{c}{V_{xy} \\ V_{wy} \\ V_{zy}}
=\matrix{ccc}{
V_{xx} & V_{xw} & V_{xz} \\
V_{wx} & V_{ww} & V_{wz} \\
V_{zx} & V_{zw} & V_{zz}}
\matrix{c}{K_x^\top \\ K_w^\top \\ K_z^\top },
\end{align*}
and, as a result,
\begin{align*}
\matrix{c}{K_x^\top \\ K_w^\top \\ K_z^\top }
=\matrix{ccc}{
V_{xx} & V_{xw} & V_{xz} \\
V_{wx} & V_{ww} & V_{wz} \\
V_{zx} & V_{zw} & V_{zz}}^{-1}
\matrix{c}{V_{xy} \\ V_{wy} \\ V_{zy}},
\end{align*}
where the inverse exists thanks to Assumption~\ref{assum:Gaussian}. Introducing the change of variable 
$$
\xi=\matrix{ccc}{
V_{xx} & V_{xw} & V_{xz} \\
V_{wx} & V_{ww} & V_{wz} \\
V_{zx} & V_{zw} & V_{zz}}^{-1}
\matrix{c}{V_{xy} \\ V_{wy} \\ V_{zy}}
$$
in~\eqref{eqn:optim:1} yields the equivalent optimization problem
\begin{subequations} \label{eqn:optim:2}
\begin{align}
\mathcal{X}=\argmin_{\xi\in\mathbb{R}^{(n_x+n_y+n_z)\times n_y}} & \,\, \trace\left( \xi^\top Z' \xi\right) \\
\mathrm{s.t.} \hspace{.4in} & \,\, \xi^\top Q' \xi\leq I,
\end{align}
\end{subequations}
where
\begin{align*}
Z'
&=\matrix{ccc}{
V_{xx} & V_{xw} & V_{xz} \\
V_{wx} & V_{ww} & V_{wz} \\
V_{zx} & V_{zw} & V_{zz}}Z\matrix{ccc}{
V_{xx} & V_{xw} & V_{xz} \\
V_{wx} & V_{ww} & V_{wz} \\
V_{zx} & V_{zw} & V_{zz}}\\
&=
W^\top \matrix{cc}{-I & 0 \\ 0 & \delta I}
W
\end{align*}
and
\begin{align*}
Q'
&=\matrix{ccc}{
V_{xx} & V_{xw} & V_{xz} \\
V_{wx} & V_{ww} & V_{wz} \\
V_{zx} & V_{zw} & V_{zz}}Q\matrix{ccc}{
V_{xx} & V_{xw} & V_{xz} \\
V_{wx} & V_{ww} & V_{wz} \\
V_{zx} & V_{zw} & V_{zz}}\\
&=
W^\top \hspace{-.05in}\left(
\matrix{cc}{V_{xx} & V_{xw} \\ V_{wx} & V_{ww} }-
\matrix{c}{V_{xz} \\ V_{wz}}V_{zz}^{-1}\matrix{c}{V_{xz} \\ V_{wz}}
\right)^{\hspace{-.03in}-1}\hspace{-.07in} W
\end{align*}
with
\begin{align*}
W=\matrix{ccc}{
V_{xx}-V_{xz}V_{zz}^{-1}V_{zx} & V_{xw}-V_{xz}V_{zz}^{-1}V_{zw} & 0\\
V_{wx}-V_{wz}V_{zz}^{-1}V_{zx} & V_{ww}-V_{wz}V_{zz}^{-1}V_{zw} & 0}.
\end{align*}
Let us define
\begin{subequations} 
\begin{align}
\mathcal{X}'=\argmin_{\xi'\in\mathbb{R}{(n_x+n_y)\times n_y}} & \,\, \trace\left( \xi^{'\top} Z'' \xi'\right) \\
\mathrm{s.t.} \hspace{.34in} & \,\, \xi^{'\top} Q'' \xi'\leq I,
\end{align}
\end{subequations}
where
\begin{align*}
Z''
&=
\matrix{cc}{
V_{xx}-V_{xz}V_{zz}^{-1}V_{zx} & V_{xw}-V_{xz}V_{zz}^{-1}V_{zw} \\
V_{wx}-V_{wz}V_{zz}^{-1}V_{zx} & V_{ww}-V_{wz}V_{zz}^{-1}V_{zw}
}\\&\hspace{.2in}\times
\matrix{cc}{-I & 0 \\ 0 & \delta I}
\\&\hspace{.2in}\times
\matrix{cc}{
V_{xx}-V_{xz}V_{zz}^{-1}V_{zx} & V_{xw}-V_{xz}V_{zz}^{-1}V_{zw}\\
V_{wx}-V_{wz}V_{zz}^{-1}V_{zx} & V_{ww}-V_{wz}V_{zz}^{-1}V_{zw} 
}
\end{align*}
and
\begin{align*}
Q''
&=
\matrix{cc}{
V_{xx}-V_{xz}V_{zz}^{-1}V_{zx} & V_{xw}-V_{xz}V_{zz}^{-1}V_{zw} \\
V_{wx}-V_{wz}V_{zz}^{-1}V_{zx} & V_{ww}-V_{wz}V_{zz}^{-1}V_{zw} 
}\\&\hspace{.2in}\times
\left(
\matrix{cc}{V_{xx} & V_{xw} \\ V_{wx} & V_{ww} }-
\matrix{c}{V_{xz} \\ V_{wz}}V_{zz}^{-1}\matrix{c}{V_{xz} \\ V_{wz}}
\right)^{-1}
\\ &\hspace{.2in}\times
\matrix{cc}{
V_{xx}-V_{xz}V_{zz}^{-1}V_{zx} & V_{xw}-V_{xz}V_{zz}^{-1}V_{zw} \\
V_{wx}-V_{wz}V_{zz}^{-1}V_{zx} & V_{ww}-V_{wz}V_{zz}^{-1}V_{zw} 
}\\
&=
\matrix{cc}{V_{xx} & V_{xw} \\ V_{wx} & V_{ww} }-
\matrix{c}{V_{xz} \\ V_{wz}}V_{zz}^{-1}\matrix{c}{V_{xz} \\ V_{wz}}.
\end{align*}
Evidently, 
\begin{align*}
\left\{\matrix{c}{\xi' \\ 0_{n_z\times n_y}} \bigg| \xi'\in\mathcal{X}'  \right\}\subseteq \mathcal{X},
\end{align*}
because
\begin{align*}
\xi^{'\top} Z'' \xi'&=\matrix{c}{\xi' \\ \xi''}Z'\matrix{c}{\xi' \\ \xi''},\\
\xi^{'\top} Q'' \xi'&=\matrix{c}{\xi' \\ \xi''}Q'\matrix{c}{\xi' \\ \xi''},
\end{align*}
for all $\xi''\in\mathbb{R}^{n_z\times n_y}$. Hence, there is an equilibrium for which $K_z=0$ because $\mathcal{X}'\neq \emptyset$ (as its cost function is continuous and the feasible set is compact). The rest of the proof follows easily from the fact that once we use the sender's strategy in the statement of theorem (which is an affine function in the realization of the random variables $x,w$), the best estimator that the receiver and the malicious agent can utilize is the least mean square estimator in the statement of the theorem.
\end{proof}

\begin{remark} Note that the policies $\psi^*$ and $\phi^*$ at the equilibrium constructed in Theorem~\ref{tho:1} are affine in the realizations of the random variables $y$ and $z$. However, these policies are overall not affine as $V_{xy}$ , $V_{xz}$, $V_{yz}$, $V_{yy}$, and $V_{zz}$ are all functions of the distributions of the random variables $z$ and $y$ (albeit not their realizations). This also shows that $\psi^*$ and $\phi^*$, at the equilibrium, are functions of $\gamma$ (i.e., the policy of the sender).
\end{remark}

Solving the optimization problem in Theorem~\ref{tho:1} is a tedious task even numerically as the problem is not convex. Note that since we are interested in finding an equilibrium, approximations are not particularly useful (since, at an approximate solution, an agent might have an incentive to unilaterally change its policy and, thus, the equilibrium breaks down). In such cases, the definition of the equilibrium should be adapted accordingly. This problem can be an avenue for future research. In the next corollary, we focus on the case where we can solve this optimization problem explicitly.

\begin{corollary} \label{cor:1} Let $n_y=1$. For any $\delta\in\mathbb{R}_{\geq 0}$, there exists an equilibrium $(\psi^*,\phi^*,\gamma^*)\in\Psi^\Gamma\times \Phi^\Gamma\times \Gamma$ in which
\begin{enumerate}
\item The receiver employs the policy $\psi^*$ such that
\begin{align*}
\mathbb{P}\left\{\hat{x}=\matrix{cc}{V_{xy} & V_{xz}}\matrix{cc}{V_{yy} & V_{yz} \\ V_{zy} & V_{zz}}^{-1}\matrix{c}{y \\ z}\right\}=1;
\end{align*}
\item The malicious agent employs the policy $\phi^*$ such that
\begin{align*}
\mathbb{P}\left\{\hat{w}=\matrix{cc}{V_{wy} & V_{wz}}\matrix{cc}{V_{yy} & V_{yz} \\ V_{zy} & V_{zz}}^{-1}\matrix{c}{y \\ z}\right\}=1;
\end{align*}
\item The sender employs the policy $\gamma^*$ such that
\begin{align*}
\mathbb{P}\{y=K_x^* x+K_w^* w\}=1,
\end{align*}
where 
\begin{align*}
\matrix{c}{K_x \\ K_w}=\Xi^{-1/2}\xi
\end{align*}
with $\xi$ being the normalized eigenvector corresponding to the smallest eigenvalue of 
$\Xi^{1/2}\diag(-I,\delta I)\Xi^{1/2}$.
\end{enumerate}
Furthermore, for all $\kappa\neq 0$, $(\psi^*,\phi^*,\kappa\gamma^*)$ constitutes an equilibrium.
\end{corollary}

\begin{proof} The proof follows from applying Lemma~A.1 in~\cite{FarokhiTAC2014Submitted} to the results of Theorem~\ref{tho:1}.
\end{proof}

\begin{remark} Corollary~\ref{cor:1} shows the maybe surprising fact that the privacy-preserving policy of the sender at the calculated family of equilibria is deterministic if the message from the sender is scalar, i.e., $n_y=1$. It should be noted that the equilibrium in this game is not necessarily unique. Therefore, there might exist an equilibrium (in addition to the babbling equilibria; see Theorem~\ref{tho:babbling}) that consists of randomization. Unfortunately, for the rest of the cases, there is no systematic way to construct an explicit policy for the sender at the equilibrium. 
\end{remark}

\section{Numerical Example} \label{sec:example}
Let us consider the case where there is no side information. Set $n_x=n_w=n_y=1$ and fix 
$$
\matrix{cc}{V_{xx} & V_{xw} \\ V_{wx} & V_{ww} }=\matrix{cc}{1 & 0.8 \\ 0.8 & 1}. 
$$ 
Fig.~\ref{fig:esterror} shows the estimation error of the receiver and the malicious agent at the calculated equilibrium with $\kappa=1$ as a function of the privacy ratio $\delta$. Indeed, as the privacy ratio $\delta$ increases, the estimation error of the malicious agent increases and, finally, converges to $V_{ww}$. Therefore, the sender keeps $w$ completely private. However, this happens at the price of a larger estimation error of the state at the receiver.

\begin{figure}
\centering
\includegraphics[width=1.0\linewidth]{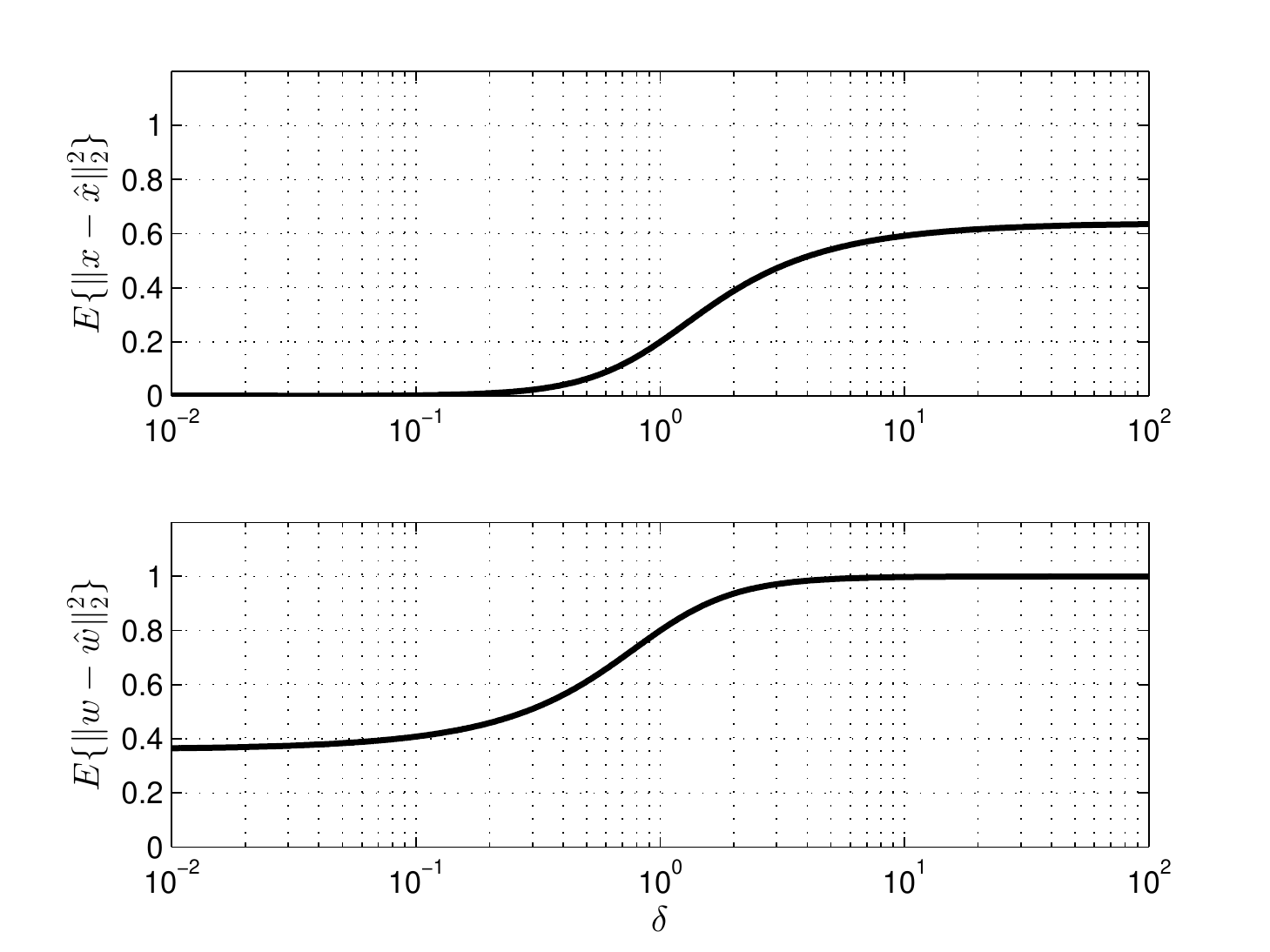}
\caption{\label{fig:esterror} Estimation error of the receiver and the malicious agent as a function of the privacy ratio. }
\end{figure}

\begin{figure}[t]
\centering
\includegraphics[width=1.0\linewidth]{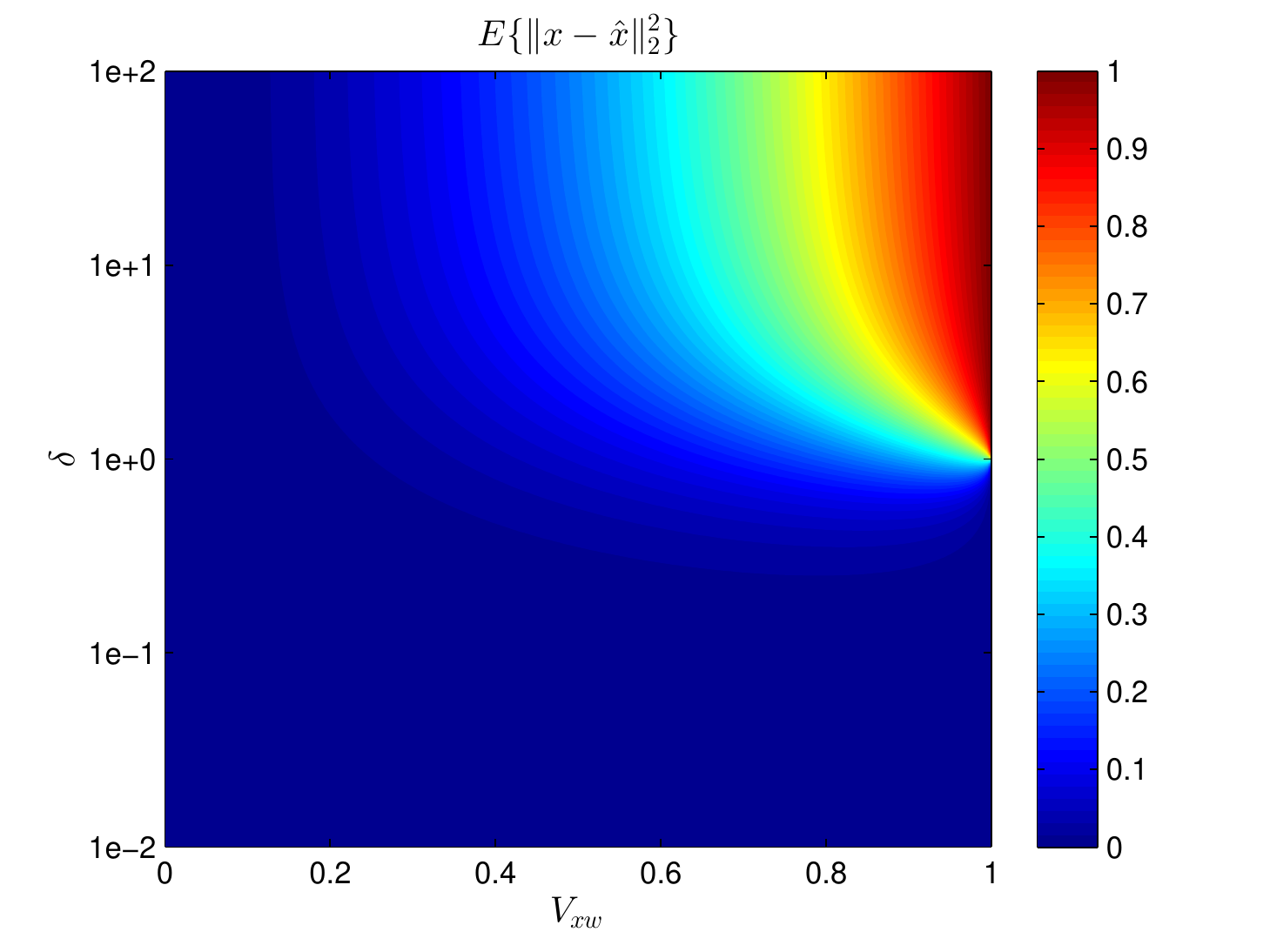}
\caption{\label{fig:contour} Estimation error of the receiver as a function of the privacy ratio and the correlation between the state and the private information.} 
\end{figure}


Fig.~\ref{fig:contour} illustrates the estimation error at the receiver $\mathbb{E}\{\|x-\hat{x}\|_2^2\}$ as a function of the privacy ratio $\delta$ and the correlation between the state and the private information captured by $V_{xw}$. For a fixed privacy ratio, as we increase $V_{xw}$ the quality of the estimation degrades since if the correlation is high an accurate measurement of $x$ conveys a lot of information about the realization of $w$. 

\section{Conclusions and Future Work} \label{sec:conclusions}
A game-theoretic framework for modelling and analysing the effect of privacy in strategic communication between various entities is developed. A well-informed sender wants to provide an accurate measurement of the state-to-be-estimated to a receiver. However, the sender also desires to keep its private information private from a malicious agent. The private information and the state are assumed to be correlated and, hence, an accurate measurement of the state  also shines some light on the realization of the private information. The trade-off between the desire to keep privacy and transmitting a useful measurement to the receiver is captured using the privacy ratio. For scalar messages, the policy of the sender, at the calculated equilibrium, is deterministic. A viable avenue for future work is to extend the setup to more than one sender.

\bibliographystyle{ieeetr}
\bibliography{ref}
\end{document}